\documentclass[11pt]{article}
\usepackage{graphicx}
\usepackage{amsmath}
\usepackage{amssymb}
\usepackage{amsfonts}
\usepackage{amsthm}
\usepackage{xcolor}
\usepackage{enumitem}
\usepackage{placeins}
\usepackage{hyperref}
\numberwithin{equation}{section}
\def\R{\mathbb{R}}
\def\P{\mathbb{P}}

\def\S{\mathcal{S}_U}
\def\SD{\mathcal{S}_{\nabla U}}
\def\G{\mathcal{G}}

\def\neweq#1{\begin{equation}\label{#1}}
\def\endeq{\end{equation}}
\def\eq#1{(\ref{#1})}
\def\eps{\varepsilon}

\def\eps{{\varepsilon}}

\hypersetup{
    colorlinks=true,
    linkcolor=blue,
    filecolor=magenta,
    urlcolor=cyan,
}

\newtheorem{theorem}{Theorem}[section]

\newtheorem{lemma}[theorem]{Lemma}

\newtheorem{definition}[theorem]{Definition}

\title{Finite-time blow-up for the forced 2D Navier-Stokes equations}
\author{Filippo GAZZOLA\footnote{Dipartimento di Matematica (Dipartimento di Eccellenza MUR 2023-2027)\ -- Politecnico di Milano
-- Piazza Leonardo da Vinci 32, 20133, Milano, Italy -- filippo.gazzola@polimi.it}}
\date{}
\begin{document}
\maketitle

\begin{abstract}
We provide explicit examples showing that the existing assumptions in literature on the force are sharp for the stated properties
of strong solutions to the 2D Navier-Stokes equations. If they fail to be satisfied, finite time blow-up may occur in different forms.
\end{abstract}

{\small
}

\baselineskip13pt

\section{Introduction}

In the entire plane, the forced evolution Navier-Stokes equations read
\begin{equation}\label{ns}
U_t-\Delta U+(U\cdot\nabla)U+\nabla P=f\, ,\quad\nabla\cdot U=0\quad\mbox{ in }\R^2\times\R_+
\end{equation}
and are complemented with an initial condition (here and in the sequel, we denote $\xi=(x,y)\in\R^2$)
\begin{equation}\label{ic}
U(\xi,0)=U_0(\xi)\quad \mbox{in }\R^2.
\end{equation}
The finite time blow-up of the solutions to this problem in $\R^3$ was analysed in \cite{GaGa}. In dependence of the integrability of $f$,
different blow-up conditions were considered, all of them complementing well-known no-blow-up statements. The main idea was the
construction of solutions inspired to the Talenti bubbles \cite{talenti} in the critical Lane-Emden semilinear elliptic equation with
a time-dependent concentration parameter.\par
In $\R^2$ several changes have to be expected, both concerning the equations and the functional analytic background.
The most important differences are existence, smoothness and global uniqueness results for Leray-Hopf solutions, regardless of additional
constraints \cite{farwig,Galdi-evol,ladybook,leray,lions,sohr} for smooth $f$. More recently, in \cite{aaffgg} the enstrophy blow-up of (3D) space-periodic solutions
to \eq{ns}-\eq{ic} (for $f=0$) appeared to be a consequence of the energy equipartition among modes in each spatial cube, but it was also shown that
equipartition of energy does not appear for 2D flows, preventing the enstrophy blow-up. Another crucial difference is the lack of interest
in the Lane-Emden equation since the embedding $\mathcal{D}^{1,2}(\R^2)\subset L^q(\R^2)$ has no critical exponent; this is also the reason why
an assumption such as $f\in L^2_{\rm loc}(\R_+;H^{-1}(\R^2))$ appears quite delicate.\par
These remarks, together with the next (known) statements by Ladyzhenskaya \cite{ladybook}, motivate the present paper.

\begin{theorem}\label{sohr2D}{\bf (Sohr \cite[Theorem 4.2.1, p.344]{sohr}).}
Let $U_0\in L^2_\sigma(\R^2)$ and $f\in L^1_{\rm loc}(\R_+;L^2(\R^2))$. Then \eqref{ns}-\eqref{ic} admits a unique weak solution
in $\R^2\times\R_+$ which, moreover, is strong.
\end{theorem}

Here and in the sequel, $L^2_\sigma(\R^2)$ denotes the subspace of divergence-free vector fields in $L^2(\R^2)$, whereas for the characterisation
of weak/strong solutions, we refer to Definition \ref{weakstrong} below. Roughly speaking, strong solutions are weak solutions with an additional
integrability property. It is our first purpose to show that the statement in Theorem \ref{sohr2D} is sharp, since finite time blow-up
may occur in different forms. This may appear surprising in view of the above mentioned no-blow-up statements for smooth $f$:
hence, the force itself may induce blow-up for $V$. A ``dual conclusion'' is reached in connection with

\begin{theorem}\label{sohr2D2}{\bf (Sohr \cite[Theorem 1.8.3, p.304]{sohr}).}
Let $U_0\in H^1_\sigma(\R^2)$ and $f\in L^2_{\rm loc}(\R_+;L^2(\R^2))$. Then \eqref{ns}-\eqref{ic} admits a unique weak solution
in $\R^2\times\R_+$ which is strong and has locally bounded enstrophy.
\end{theorem}

Here, $H^1_\sigma(\R^2)$ denotes the subspace of divergence-free vector fields in $H^1(\R^2)$, whereas for the definition of enstrophy,
we refer again to Definition \ref{weakstrong} below. We will prove that if the assumption on time-integrability of $f$ is weakened, then the
enstrophy may blow up in finite time.\par
This paper is organised as follows. In Section \ref{mainresults} we state our main results and we comment them in comparison to Theorems
\ref{sohr2D} and \ref{sohr2D2}. In Section \ref{prelim} we construct the solutions by ``playing'' with some parameters.
Sections \ref{planar} and \ref{planar2} contain the proofs of the main results: they are obtained through a delicate choice of the
parameters, which need to obey suitable constraints.

\section{Main results}\label{mainresults}

For $p\ge1$, let $L^p_\sigma(\R^2)$ be the subspace of $L^p(\R^2)$ of divergence-free vector fields.
The initial velocity $U_0$ in \eq{ic} will be taken in the following space of smooth functions
$$
\G_\sigma(\R^2)=\{U\in\G;\ \nabla\cdot U=0\mbox{ in }\R^2\}\quad\mbox{where}\quad
\G\, =\, \bigcap_{k\in\mathbb{N}}\,\bigcap_{q\in[2,\infty]}\,W^{k,q}(\R^2)\ \subset C^\infty(\R^2)\, .
$$

Let us define weak and strong solutions.

\begin{definition}\label{weakstrong}
Let $\mathcal{D}=\{\phi\in C^\infty_c\big(\R^2\times\R_+\big)\mbox{ s.t. }\nabla\cdot\phi=0\mbox{ in }\R^2\times\R_+\}$.
For given $U_0\in L^2_\sigma(\R^2)$ and $f\in L^1_{\rm loc}(\R_+;L^2(\R^2))$, a vector field $U=U(\xi,t)$ is a
global weak solution to \eqref{ns}-\eqref{ic} in $\R^2\times\R_+$ if
\neweq{functionalcond}
U\in L^\infty_{\rm loc}(\R_+;L^2_\sigma(\R^2))\cap L^2_{\rm loc}(\R_+;H^1(\R^2))\, ,
\endeq
\neweq{weakform}
\int_{0}^{\infty}\left\{\int_{\R^2}\big[\nabla V:\nabla\phi+(V\cdot\nabla)V\cdot\phi-V\cdot\phi_t\big]\right\}\, dt=
\int_{0}^{\infty}\int_{\R^2} f\, \phi+\int_{\R^2}U_0\cdot\phi(0)\qquad\forall\phi\in\mathcal{D}\, .
\endeq
If $U\in L^4_{\rm loc}(\R_+;L^4_\sigma(\R^2))$, then $U$ is called a strong solution to \eqref{ns}-\eqref{ic} in $\R^2\times\R_+$ and,
in such case, the following energy identity holds
\neweq{energyineq}
\|U(t)\|_{L^2(\R^2)}^2+2\int_{0}^{t}\|\nabla U(s)\|_{L^2(\R^2)}^2ds\le\|V_0\|_{L^2(\R^2)}^2+2\int_{0}^{t}\int_{\R^2}f\, \phi
\qquad\forall t\in(0,T]\, .
\endeq
If finite, the quantity $\|\nabla U(t)\|_{L^2(\R^2)}^2$ is called the enstrophy of $U$.
\end{definition}

Bearing in mind Theorem \ref{sohr2D}, we may now state our first main result.

\begin{theorem}\label{main}
For any $\eps>0$ and $T>0$ there exists $f_\eps\in L^1_{\rm loc}(\R_+,L^2(\R^2))$ and $U_0\in\G_\sigma(\R^2)$ such that the unique strong solution
$U$ to \eqref{ns}-\eqref{ic}, as given by Theorem \ref{sohr2D}, satisfies
$$
U\not\in L^\infty(0,T;L^{2+\eps}(\R^2))\, ,\quad U\not\in L^{4+\eps}(0,T;L^4(\R^2))\, ,\quad U\not\in L^4(0,T;L^{4+\eps}(\R^2))\, ,
$$
$$
\nabla U\not\in L^2(0,T;L^{2+\eps}(\R^2))\, ,\quad \nabla U\not\in L^{2+\eps}(0,T;L^2(\R^2))\, .
$$
\end{theorem}

In particular, the only $L^p$-norm that remains locally bounded is the $L^2$-norm. The effective integrability of a weak solution
$U\in L^r_{\rm loc}(\R_+;L^q(\R^2))$ is measured by the Serrin number \cite{serrin2}
$$\S=\frac{2}{r}+\frac{2}{q}\, ,$$
see also the systematic use of it in \cite{farwig}. On the one hand, combined with \eq{energyineq}, the condition $\S\le1$ ensures uniqueness
of weak solutions \cite[Theorem 1.5.1, p.276]{sohr}; see also earlier work in \cite{giga,kozono,masuda,serrin2}. On the other hand, the
``break even'' is obtained when $r=q=4$, yielding $\S=1$, see Theorem \ref{sohr2D}. Theorem \ref{main} states that $\S<1$ may not be reached
under the sole assumption $f\in L^1_{\rm loc}(\R_+,L^2(\R^2))$. Based on the seminal 3D work by Beir{\~a}o da Veiga \cite{Hugo},
Galdi \cite[Remark 5.3]{Galdi-evol} (see also \cite[p.3248]{farwig}) considers the Serrin number for $\nabla U\in L^r_{\rm loc}(\R_+;L^q(\R^2))$:
it turns out that, if $\SD=2$ plus an additional assumption, then existence and uniqueness are ensured. The break even is now obtained when $r=q=2$,
as in Theorem \ref{sohr2D}. Theorem \ref{main} states that $\SD<2$ may not be reached under the sole assumption $f\in L^1_{\rm loc}(\R_+,L^2(\R^2))$.
All this shows that Theorem \ref{sohr2D} is sharp: {\em no additional integrability on the solution and its gradient have to be expected when}
$f\in L^1_{\rm loc}(\R_+,L^2(\R^2))$.\par
Our second main result, to be compared with Theorem \ref{sohr2D2}, reads

\begin{theorem}\label{main2}
For any $\eps\in(0,1]$ and $T>0$ there exists $f_\eps\in L^{2-\eps}_{\rm loc}(\R_+,L^2(\R^2))$ and $U_0\in\G_\sigma(\R^2)$ such that the unique
strong solution $U$ to \eqref{ns}-\eqref{ic}, as given by Theorem \ref{sohr2D2}, satisfies
$$
\lim_{t\to T}\|\nabla U(t)\|_{L^2(\R^2)}=+\infty\, .
$$
\end{theorem}

This shows that also Theorem \ref{sohr2D2} is sharp: {\em by slightly relaxing the time-integrability of the force, the enstrophy of the
solution may blow up in finite time}.

\section{Preliminaries}\label{prelim}

We begin with an extremely useful calculus statement.

\begin{lemma}\label{calculus}
Let $\alpha,\gamma>0$, $k\in\mathbb{N}$, $p>k$. For the function $h:\R^2\times[0,T)\to\R$ defined by
$$
h(\xi,t)=\frac{(T-t)^{2\alpha}\, x^k}{\left[(T-t)^{2\gamma}+|\xi|^2\right]^{p/2}}\qquad\forall(\xi,t)\in\R^2\times[0,T)
$$
we have that
\neweq{general}
\forall q>\max\left\{1,\frac{2}{p-k}\right\}\quad\exists C_q>0\quad\mbox{s.t.}\quad\|h(t)\|_{L^q(\R^2)}^q
=C_q(T-t)^{2q\alpha+(kq+2-pq)\gamma}\quad\forall t\in[0,T)\, .
\endeq
Hence,
\neweq{iff2}
\begin{array}{rcl}
\forall r\ge1\qquad h\in L^r(0,T;L^q(\R^2)) & \Longleftrightarrow & 2q\alpha+(kq+2-pq)\gamma>-q/r\, ,\\
h\in L^\infty(0,T;L^q(\R^2)) & \Longleftrightarrow & 2q\alpha+(kq+2-pq)\gamma\ge0\, ,\\
\lim_{t\to T}\|h(t)\|_{L^q(\R^2)}=0 & \Longleftrightarrow & 2q\alpha+(kq+2-pq)\gamma>0\, .
\end{array}
\endeq
The same statements hold if $x^k$ is replaced by any homogeneous polynomial in $x,y$ of degree $k$.\par
\end{lemma}
\begin{proof} In what follows, $C>0$ denotes suitable constants that may also vary within the same equation.
Using polar coordinates, combined with the change of variables $\rho=(T-t)^{\gamma}s$ for fixed $t\in[0,T)$, we find
\begin{align*}
\|h(t)\|_{L^q(\R^2)}^q &= (T-t)^{2q\alpha}\int_{\R^2}\frac{|x|^{kq}\ d\xi}{[(T-t)^{2\gamma}+|\xi|^2]^{pq/2}}
=C(T-t)^{2q\alpha}\int_0^\infty\frac{\rho^{kq+1}\ d\rho}{[(T-t)^{2\gamma}+\rho^2]^{pq/2}} \\
&= C(T-t)^{2q\alpha+(kq+2-pq)\gamma}\int_0^\infty\frac{s^{kq+2}\ ds}{[1+s^2]^{pq/2}}=C_q(T-t)^{2q\alpha+(kq+2-pq)\gamma}
\qquad\forall q>\tfrac{2}{p-k}\, ,
\end{align*}
which proves \eq{general}. Then also the three statements in \eq{iff2} follow.\par
Finally, after switching to polar coordinates, the computations do not change if $x^k$ is replaced by any homogeneous polynomial of degree $k$.
\end{proof}

Let us explain how we intend to build the solution to \eq{ns}-\eq{ic}, starting from the velocity $U$, determining the associated pressure $P$,
ending with the force $f$.
Instead of the velocity vector field $V$ used in \cite{GaGa}, for any $T,\delta,\gamma,\ell>0$ the natural guess here is to take
$$
U(\xi,t)=\frac{(T-t)^{2\delta}}{[(T-t)^{2\gamma}+|\xi|^2]^{\ell/2}}\, (y,-x),
$$
so that $U\in C^\infty(\R^2\times[0,T))$ is divergence-free. Then we compute
$$
\Delta U(\xi,t)=\frac{\ell(T-t)^{2\delta}[(\ell-2)|\xi|^2-2(T-t)^{2\gamma}]}{[(T-t)^{2\gamma}+|\xi|^2]^{2+\ell/2}}\, (y,-x)\, ,
\qquad[(U\cdot\nabla)U](\xi,t)=\frac{-(T-t)^{4\delta}\ \xi}{\left[(T-t)^{2\gamma}+|\xi|^2\right]^\ell}\, .
$$
On the one hand, this suggests to take $\ell=2$ but, with this choice, $U(t)\not\in L^2(\R^2)$ for any $t\in[0,T)$ and
$U$ does not reach the minimal condition $U(t)\in L^2(\R^2)$ allowing to construct a weak solution to \eq{ns} in $\R^2\times[0,T)$.
On the other hand, the magic cancellation observed in \cite{GaGa} acts again in the convection: the exponent at the denominator of
$(U\cdot\nabla)U$ {\em is not} $1+\ell$ as expected when multiplying $U$ with a first order derivative. If we replace $U$ with the more general
vector field
$$
H(\xi,t)=\frac{(T-t)^{2\delta}}{[(T-t)^{2\gamma}+|\xi|^2]^{\ell/2}}\, (y,-kx)\qquad(k\neq0)\, ,
$$
we see that
$$
\nabla\cdot H=0\, \Longleftrightarrow\, k=1\, ,\qquad (H\cdot\nabla)H=\frac{-k(T-t)^{4\delta}\, \xi}{\left[(T-t)^{2\gamma}+|\xi|^2\right]^\ell}
+\ell(k-1)\, \frac{(T-t)^{4\delta}(xy^2,-kyx^2)}{\left[(T-t)^{2\gamma}+|\xi|^2\right]^{1+\ell}}
$$
and, hence, the magic cancellation is a consequence of the divergence-free condition.\par
In order to ensure that $U\in L^\infty(0,T;L^2(\R^2))$ we need the restriction $\ell>2$. Then, as explained in \cite{GaGa}, any
$\ell$ leads to the same results. For simplicity and elegance (the powers of the denominators of $\Delta U$ and $(U\cdot\nabla)U$ coincide),
we take $\ell=4$ so that
\neweq{U2D}
U(\xi,t)=\frac{(T-t)^{2\delta}}{[(T-t)^{2\gamma}+|\xi|^2]^2}\, (y,-x),
\endeq
From now on, we denote by
\neweq{Pj}
\P_j(\xi)\mbox{ any homogeneous polynomial of degree $j\ge0$ with respect to $|\xi|$.}
\endeq
For instance, $\P_1(\xi)=|\xi|$ and $\P_1(\xi)=x$ (with an abuse of notation).\par
Using this convention, we find
\neweq{Ut}
\nabla U(\xi,t)=\frac{(T-t)^{2\delta}\P_2(\xi)+(T-t)^{2\delta+2\gamma}}{\left[(T-t)^{2\gamma}+|\xi|^2\right]^3}\, ,\qquad
U_t(\xi,t)=\frac{(T-t)^{2\delta-1}[(T-t)^{2\gamma}\P_1(\xi)+\P_3(\xi)]}{\left[(T-t)^{2\gamma}+|\xi|^2\right]^3}\, ,
\endeq
\neweq{DeltaU}
\Delta U(\xi,t)=\frac{(T-t)^{2\delta}\P_3(\xi)+(T-t)^{2\delta+2\gamma}\P_1(\xi)}{[(T-t)^{2\gamma}+|\xi|^2]^4}\, ,
\qquad[(U\cdot\nabla)U](\xi,t)=\frac{(T-t)^{4\delta}\P_1(\xi)}{\left[(T-t)^{2\gamma}+|\xi|^2\right]^4}\, .
\endeq

By using in different ways Lemma \ref{calculus}, we obtain
\neweq{ULinftyL2}
U\in L^\infty(0,T;L^2(\R^2))\ \Longleftrightarrow\ \delta\ge\gamma\, ,
\endeq
\neweq{UL2H1}
\nabla U\in L^2(0,T;L^2(\R^2))\ \Longleftrightarrow\ 4\delta+1>6\gamma\, ,
\endeq
\neweq{UL4L4}
U\in L^4(0,T;L^4(\R^2))\ \Longleftrightarrow\ 8\delta+1>10\gamma\, .
\endeq
Just to reassure the reader, we notice that \eq{ULinftyL2}-\eq{UL2H1}$\Rightarrow$\eq{UL4L4}, as expected from \cite{lions}: indeed,
$$
8\delta+1=4\delta+1+4\delta>6\gamma+4\gamma=10\gamma\, .
$$
In particular, up to replacing $(T-t)$ with $|T-t|$, \eq{ULinftyL2}-\eq{UL2H1}-\eq{UL4L4} yield the following implication
\neweq{globalocal}
\gamma<\min\left\{\delta,\frac{4\delta+1}{6}\right\}\ \Longrightarrow\ U\in L^\infty_{\rm loc}(\R_+;L^2(\R^2))\cap L^4_{\rm loc}(\R_+;L^4(\R^2))\, ,
\quad\nabla U\in L^2_{\rm loc}(\R_+;L^2(\R^2))\, .
\endeq

The next step is to prove

\begin{lemma}\label{defforce2D}
{\rm {\bf (Integrability of ${\mathbf{g}}$).}} Let $U$ be as in \eqref{U2D} and let $g:=U_t-\Delta U+(U\cdot\nabla)U$. Then
\neweq{sumg2D}
\forall(r,q)\in[1,\infty)\times(1,\infty)\quad
g\in L^r(0,T;L^q(\R^2))\ \Longleftrightarrow\ \gamma<\min\left\{\frac{2\delta+1/r}{5-2/q},\frac{4\delta+1/r}{7-2/q},
\frac{2\delta-1+1/r}{3-2/q}\right\}\, .
\endeq
In particular,
\neweq{gL1L2}
\forall r\ge1\qquad g\in L^r(0,T;L^2(\R^2))\ \Longleftrightarrow\ \gamma<\min\left\{\frac{2\delta+1/r}{4},\frac{2\delta-1+1/r}{2}\right\}\, .
\endeq
\end{lemma}
\begin{proof} For $U$ as in \eq{U2D}, from \eq{Ut} and \eq{DeltaU} we infer
$$
\Delta U\in L^r(0,T;L^q(\R^2))\ \Longleftrightarrow\ \gamma<\frac{2\delta+1/r}{5-2/q}\, ,\qquad
(U\cdot\nabla)U\in L^r(0,T;L^q(\R^2))\ \Longleftrightarrow\ \gamma<\frac{4\delta+1/r}{7-2/q}\, ,
$$
$$
U_t\in L^r(0,T;L^q(\R^2))\ \Longleftrightarrow\ \gamma<\frac{2\delta-1+1/r}{3-2/q}\, .
$$
This proves \eq{sumg2D}. When $q=2$ and $r\ge1$, \eq{sumg2D} becomes
$$
g\in L^r(0,T;L^2(\R^2))\ \Longleftrightarrow\ \gamma<\min\left\{\frac{2\delta+1/r}{4},\frac{4\delta+1/r}{6},
\frac{2\delta-1+1/r}{2}\right\}
$$
and we notice that
$$
\delta\le\frac{3}{2}-\frac{1}{r}\ \Longleftrightarrow\ \frac{2\delta-1+1/r}{2}\le\frac{4\delta+1/r}{6}\, ,\qquad
\delta\ge\frac{1}{2r}\ \Longleftrightarrow\ \frac{4\delta+1/r}{6}\ge\frac{2\delta+1/r}{4}\, .
$$
Since $\tfrac{1}{2r}\le\tfrac{3}{2}-\tfrac{1}{r}$ for all $r\ge1$, we obtain \eq{gL1L2}.
\end{proof}

We now introduce the pressure $P$ and the resulting force $f$.
Since our purpose is only to exhibit blow-up (in different forms) for $U$, we {\em fix} any $P\in C^1(\R^2\times(0,T))$ such that
$\nabla P\in L^2(0,T;L^2(\R^2))$ and we put $f:=g+\nabla P$. Then, from Lemma \ref{defforce2D} we infer that
$f\in L^r(0,T;L^2(\R^2))$ whenever \eq{gL1L2} holds.

\section{Proof of Theorem \ref{main}}\label{planar}

By imposing the existence conditions \eq{ULinftyL2}-\eq{UL2H1} at the same time as \eq{gL1L2} with $r=1$, the resulting upper bound
for $\gamma$ reads
\neweq{uppergamma}
\gamma<\min\left\{\delta,\frac{4\delta+1}{6},\frac{8\delta+1}{10},\frac{2\delta+1}{4}\right\}=
\min\left\{\delta,\frac{2\delta+1}{4}\right\}\, .
\endeq

Using again Lemma \ref{calculus}, the blow-up conditions to be shown in Theorem \ref{main} become
\neweq{blupLqin2D}
\forall p>2\qquad\lim_{t\to T}\|U(t)\|_{L^p(\R^2)}=+\infty\ \Longleftrightarrow\ \frac{2p\delta}{3p-2}<\gamma\, ,
\endeq
\neweq{blupU4p}
\forall p>2\qquad\|U\|_{L^4(0,T;L^{2p}(\R^2))}=+\infty\ \Longleftrightarrow\ \frac{2\delta+1/4}{3-1/p}\le\gamma\, ,
\endeq
\neweq{blupUp4}
\forall p>2\qquad\|U\|_{L^{2p}(0,T;L^4(\R^2))}=+\infty\ \Longleftrightarrow\ \frac{4p\delta+1}{5p}\le\gamma\, ,
\endeq
\neweq{blupgradin2D}
\forall p>2\qquad\|\nabla U\|_{L^2(0,T;L^p(\R^2))}=+\infty\ \Longleftrightarrow\ \frac{\delta+1/4}{2-1/p}\le\gamma\, ,
\endeq
\neweq{blupgradin2D2}
\forall p>2\qquad\|\nabla U\|_{L^p(0,T;L^2(\R^2))}=+\infty\ \Longleftrightarrow\ \frac{2p\delta+1}{3p}\le\gamma\, .
\endeq

We prove that, together with \eq{uppergamma}, these blow-up conditions determine triangles in the $(\delta,\gamma)$-plane.

\begin{lemma}\label{conditions}
Combined with \eqref{uppergamma}, each of the conditions \eqref{blupLqin2D}, \eqref{blupU4p}, \eqref{blupUp4}, \eqref{blupgradin2D},
\eqref{blupgradin2D2}, defines a triangle in the first quadrant of the $(\delta,\gamma)$-plane.
\end{lemma}
\begin{proof}
\underline{Condition \eq{blupLqin2D}.} If we impose \eq{blupLqin2D} at the same time as \eq{uppergamma}, the resulting constraints read
\neweq{primoblup}
(\delta,\gamma)\in\R_+^2\, ,\qquad\frac{2p\delta}{3p-2}<\gamma<\min\left\{\delta,\frac{2\delta+1}{4}\right\}\, .
\endeq
Whatever $p>2$ is, the interval for $\gamma$ defined in \eq{primoblup} is nonempty and the region defined by \eq{primoblup} the
(nonempty, bounded) open triangle having vertices at
$$
(\delta,\gamma)\in\left\{\Big(0,0\Big);\Big(\tfrac12,\tfrac12\Big);\Big(\tfrac{3p-2}{2p+4},\tfrac{p}{p+2}\Big)\right\}\, .
$$

\underline{Condition \eq{blupU4p}.} If we impose \eq{blupU4p} (with strict inequality) at the same time as \eq{uppergamma}
the resulting constraints read
\neweq{secondoblup}
(\delta,\gamma)\in\R_+^2\, ,\qquad\frac{2\delta+1/4}{3-1/p}<\gamma<\min\left\{\delta,\frac{4\delta+1}{6}\right\}\, .
\endeq
Whatever $p>2$ is, the interval for $\gamma$ defined in \eq{secondoblup} is nonempty and the region in \eq{secondoblup} is the (nonempty, bounded)
open triangle having vertices at
$$
(\delta,\gamma)\in\left\{\Big(\tfrac{p}{4p-4},\tfrac{p}{4p-4}\Big);\Big(\tfrac12,\tfrac12\Big);\Big(\tfrac{3p-2}{8},\tfrac{p}{4}\Big)\right\}\, .
$$

\underline{Condition \eq{blupUp4}.} If we impose \eq{blupU4p} (with strict inequality) at the same time as \eq{uppergamma}
the resulting constraints read
\neweq{terzoblup}
(\delta,\gamma)\in\R_+^2\, ,\qquad\frac{4p\delta+1}{5p}<\gamma<\min\left\{\delta,\frac{4\delta+1}{6}\right\}\, .
\endeq
Whatever $p>2$ is, the interval for $\gamma$ defined in \eq{terzoblup} is nonempty and the region in \eq{terzoblup} is the (nonempty, bounded)
open triangle having vertices at
$$
(\delta,\gamma)\in\left\{\Big(\tfrac{p}{4p-4},\tfrac{p}{4p-4}\Big);\Big(\tfrac12,\tfrac12\Big);\Big(\tfrac{5p-6}{4p},\tfrac{p-1}{p}\Big)\right\}\, .
$$

\underline{Condition \eq{blupgradin2D}.} If we impose \eq{blupgradin2D} (with strict inequality) at the same time as \eq{uppergamma}
the resulting constraints read
\neweq{second2D}
(\delta,\gamma)\in\R_+^2\, ,\qquad\frac{\delta+1/4}{2-1/p}<\gamma<\min\left\{\delta,\frac{4\delta+1}{6}\right\}\, .
\endeq
Whatever $p>2$ is, the interval for $\gamma$ defined in \eq{second2D} is nonempty and the region in \eq{second2D} is the (nonempty, bounded)
open triangle having vertices at
$$
(\delta,\gamma)\in\left\{\Big(\tfrac{p}{4p-4},\tfrac{p}{4p-4}\Big);\Big(\tfrac12,\tfrac12\Big);\Big(\tfrac{p-1}{2},\tfrac{p}{4}\Big)\right\}\, .
$$

\underline{Condition \eq{blupgradin2D2}.} If we impose \eq{blupgradin2D2} (with strict inequality) at the same time as \eq{uppergamma}
the resulting constraints read
\neweq{third2D}
(\delta,\gamma)\in\R_+^2\, ,\qquad\frac{2\delta+1/p}{3}<\gamma<\min\left\{\delta,\frac{4\delta+1}{6}\right\}\, .
\endeq
Whatever $p>2$ is, the interval for $\gamma$ defined in \eq{third2D} is nonempty and the region \eq{third2D} is the (nonempty, bounded)
open triangle having vertices at
$$
(\delta,\gamma)\in\left\{\Big(\tfrac{1}{p},\tfrac{1}{p}\Big);\Big(\tfrac12,\tfrac12\Big);\Big(\tfrac{3p-4}{2p},\tfrac{p-1}{p}\Big)\right\}\, .
$$

Therefore, all the five regions are nonempty open triangles.\end{proof}

\begin{proof}[Proof of Theorem \ref{main}] For any $p>2$, all the five triangles defined in Lemma \ref{conditions} contain the open segment
with endpoints
$$
(\delta,\gamma)\in\left\{\left(\frac{1}{2},\frac{5p}{4(3p-1)}\right);\left(\frac{1}{2},\frac{1}{2}\right)\right\}\, .
$$
Hence, the intersection of the five triangles \eq{primoblup}-\eq{secondoblup}-\eq{terzoblup}-\eq{second2D}-\eq{third2D}, characterised by
$$
\max\left\{\frac{2p\delta}{3p-2},\frac{2\delta+1/4}{3-1/p},\frac{4p\delta+1}{5p},\frac{\delta+1/4}{2-1/p},\frac{2\delta+1/p}{3}\right\}
<\gamma<\min\left\{\delta,\frac{2\delta+1}{4}\right\}\, ,
$$
is nonempty and convex: its closure degenerates to the sole point $(\delta,\gamma)=(\tfrac12,\tfrac12)$ as $p\to2$.
Whatever $p>2$ is, if $(\delta,\gamma)\subset\R_+^2$ belongs to this (nonempty and convex) polygon,
then the divergence-free vector field in \eq{U2D} satisfies \eq{ULinftyL2}-\eq{UL2H1}-\eq{UL4L4}; moreover, $f\in L^1(0,T;L^2(\R^2))$.
Furthermore, all the five blow-up conditions in \eq{blupLqin2D}-\eq{blupU4p}-\eq{blupUp4}-\eq{blupgradin2D}-\eq{blupgradin2D2} hold.
For any $\eps>0$ it then suffices to take $p>2$ sufficiently close to $2$ and $(\delta,\gamma)$ in the above polygon to complete
the proof.\end{proof}

\section{Proof of Theorem \ref{main2}}\label{planar2}

By imposing the existence conditions \eq{ULinftyL2}-\eq{UL2H1} at the same time as \eq{gL1L2} with $r=2-\eps$, the resulting upper bound
for $\gamma$ reads
$$
\gamma<\min\left\{\delta,\frac{4\delta+1}{6},\frac{2\delta+1/(2-\eps)}{4},\frac{2\delta-1+1/(2-\eps)}{2}\right\}=
\min\left\{\frac{2\delta+1/(2-\eps)}{4},\frac{2\delta-1+1/(2-\eps)}{2}\right\}\, .
$$
Using \eq{Ut} and again Lemma \ref{calculus}, the blow-up condition to be shown in Theorem \ref{main2} becomes
$$
\lim_{t\to T}\|\nabla U(t)\|_{L^2(\R^2)}=C\, \lim_{t\to T}(T-t)^{2\delta-3\gamma}=+\infty\ \Longleftrightarrow\ \frac{2\delta}{3}<\gamma\, .
$$
These two conditions hold at the same time, whenever
\neweq{blupenstrof}
(\delta,\gamma)\in\R_+^2\, ,\qquad\frac{2\delta}{3}<\gamma<\min\left\{\frac{2(2-\eps)\delta+1}{4(2-\eps)},
\frac{2(2-\eps)\delta-1+\eps}{2(2-\eps)}\right\}\, .
\endeq
These inequalities define the open triangle in the $(\delta,\gamma)$-plane having vertices at
$$
(\delta,\gamma)\in\left\{\left(\frac{3(1-\eps)}{2(2-\eps)},\frac{1-\eps}{2-\eps}\right);\left(\frac{3-\eps}{2(2-\eps)},\frac{4-\eps}{4(2-\eps)}\right);
\left(\frac{3}{2(2-\eps)},\frac{1}{2-\eps}\right)\right\}
$$
Hence, for all $\eps\in(0,1]$ the triangle is open and nonempty: its closure degenerates to the sole point $(\delta,\gamma)=(\tfrac34,\tfrac12)$
as $\eps\to0$. For any $\eps\in(0,1]$ take $(\delta,\gamma)$ satisfying \eq{blupenstrof} to conclude the proof of Theorem \ref{main2}.
\par\bigskip\noindent
{\bf Data availability statement.} Data sharing not applicable to this article as no datasets were generated or analysed during the current study.
\par\noindent
{\bf Conflict of interest statement}. The Author declares that he has no conflict of interest.
\par\noindent
{\bf Acknowledgements.} The Author is supported by the MUR grant {\em Dipartimento di Eccellenza 2023-27} (Italy) and by INdAM.

\end{document}